\documentclass[oneside,10pt]{article}          
\usepackage[b5paper]{geometry}	    
\usepackage{amsfonts,amsmath,latexsym,amssymb} 
\usepackage{theorem}
\usepackage{mathrsfs,upref}         
\usepackage{mathptmx}
		    
\newtheorem{theorem}{Theorem}           
\newtheorem{lemma}{Lemma} 
\newtheorem{proposition}{Proposition}              
\newtheorem{corollary}{Corollary}
\theoremstyle{definition}
\newtheorem{definition}{Definition}
\newtheorem{example}{Example}

\newtheorem{proof}{Proof}
\newtheorem{remark}{Remark}

\newcommand{\hil}{\mathcal{H}}
\newcommand{\angf}[2]{c\left(\,#1,\,#2\,\right)}
\newcommand{\pint}[2]{\displaystyle \left \langle #1,#2 \right\rangle}
\def\cB{\mathcal{B}}
\def\cF{\mathcal{F}}
\def\cG{\mathcal{G}}
\def\cW{\mathcal{W}}

\DeclareMathOperator{\dist}{dist}

\date{}
\begin{document}

\title{On perturbations of woven pairs of frames}

\author{P. Calder\'on \& M. Ruiz}
\maketitle

\begin{abstract}
 In this note, we prove some results related to  small perturbations of a frame for a Hilbert space $\hil$ in order to have a woven pair for $\hil$. Our results complete those known in the literature. In addition we study a necessary condition for a woven pair, that resembles a characterization for Riesz frames.
\end{abstract}

\section{Introduction and preliminaries}
Woven families of frames were introduced in \cite{BCGLL} motivated by  a problem of distributed signal processing in which the pre-processing of a signal is performed by a family of frames that correspond to  a wireless sensor network. The purpose is to  have some robustness in the reconstruction of a signal independently of the set of  measurements obtained at each node or sensor. Mathematically speaking, the idea is to have a family of frames $\{f_{ij}\}_{i\in I, j\in I_n}$ for a separable Hilbert space $\hil$  such that, for every partition $\{\sigma_j\}_{j\in I_n}$ of $I$, the set $\{f_{ij}\}_{i\in \sigma_j, j\in I_n}$ is a frame for $\hil$.

The study of woven frames continued in the work \cite{CL} where the authors reviewed some basic results relating perturbations by invertible operators, projection of woven frames, and a weaving equivalent of unconditional sequences. Weaving frames were then generalized to the context of Banach frames (\cite{CFL}), K-frames (\cite{DeV2}), continuous frames (\cite{DeV1}),  fusion frames (\cite{DeVV}, \cite{NA}),  among others. The topic attracted the interest of researchers in frame theory, as is evident from the number of papers published in recent years.

Our purpose in this note is to study the notion of woven pairs in terms of the synthesis operators of the frames involved. From this perspective, simpler proofs of perturbation results can be achieved. We also use this setting to look for characterization of woven pairs in terms of the angle between the nullspace of some operator and a  family of ranges of obliques projections. 

\subsection{Frames and woven frames.}
 
In this section, we give a brief summary of frame theory and fix some notations used throughout the paper. Let $\hil$ be a separable (finite or infinite dimensional) Hilbert space. By $B(\hil)$, we mean the algebra of bounded linear operators on $\hil$. Given an operator $T\in B(\hil)$, we denote by $R(T)$ and $N(T)$ its range and nullspace,  respectively.  Finally, we shall denote  by $I$ an index set (finite or countably infinite) and by $I_m$ the finite index set $I_m=\{1,2,\ldots,m\}$, for $m\in \mathbb{N}$. 

\begin{definition}A family $\cF=\{f_i\}_{i\in I}$ of vectors in $\hil$ is called a {\it frame} for $\hil$ if there exist constants $A,B>0$ such that
\[ A\|x\|^2 \leq \sum_{i\in I}\, |\pint{x}{f_i}|^2\leq B\|x\|^2\, , \quad \text{ for every } x\in \hil.\]
If we have only the upper bound condition, we say that $\cF$ is a {\it Bessel sequence} for $\hil$.
The maximal  lower bound $A_\cF$ and the minimal upper bound $B_\cF$   are called the {\it optimal frame bounds} of $\cF$. 
\end{definition}
Let fix some orthonormal basis $\cB=\{e_i\}_{i\in I}$ for $\hil$. Then, one can associate some bounded linear operators to a Bessel sequence $\cF$. Namely, the {\it synthesis operator} $T_{\cF}\in B(\hil)$, defined by $T_{\cF}\,e_i=f_i$, the {\it analysis operator} $T_{\cF}^*$, which is the adjoint of $T_{\cF}$, and finally, the {\it frame operator}, $S_{\cF}=T_{\cF}T_{\cF}^*$.

The frame operator plays a major role in the reconstruction of a vector $f\in \hil$ from its frame coefficients $\{\pint{x}{f_i}\}_{i\in I}$: we define the {\it canonical dual} of a frame $\cF$ as the sequence $S_\cF^{-1}(\cF)=\{S_\cF^{-1}\,f_i\}_{i\in I}$.  We have then the {\it reconstruction formulas}:
\[x=\sum_{i\in I} \pint{x}{f_i}\, S_\cF^{-1}f_i=\sum_{i\in I}\pint{x}{S_\cF^{-1} f_i}\, f_i.\]

There are several results and frame features that can be stated in terms of these operators. In the following Proposition we list some of them
\begin{proposition}\label{equiv frames} Let $\cF=\{f_i\}_{i\in I}$ be a Bessel sequence in $\hil$, with synthesis operator $T_{\cF}$. Then, the following are equivalent:
\begin{enumerate}
\item $\cF$ is a frame for $\hil$.
\item $T_{\cF}$ is a surjective operator.
\item $S_{\cF}=T_{\cF}T_{\cF}^*$ is a positive invertible operator.
\end{enumerate}
Moreover, the (optimal) frame bounds can be computed as $A_\cF=\|T_{\cF}^\dagger\|^{-2}$, $B_\cF=\|T_{\cF}\|^2$, where $T_{\cF}^\dagger$ is the Moore-Penrose pseudoinverse of $T_{\cF}$ and $\|\cdot\|$ is the usual operator norm in $B(\hil)$.
\end{proposition}

\subsubsection*{Woven frames} 
Lets recall the definition of woven frames given in \cite{BCGLL}:
\begin{definition}
A family of frames $\cF_j=\{f_{ij}\}_{i\in I}$ for $j\in I_n$  for a Hilbert space $\hil$ is said to be {\it woven} if there are universal constants $A$ and $B$ such that for every partition $\{\sigma_j\}_{j\in I_n}$ of $I$, the family $\{f_{ij}\}_{i\in \sigma_j, \, j\in I_n}$ is a frame for $\hil$, with lower and upper frame bounds $A$ and $B$, respectively. Each frame $\cF_{\sigma}=\{f_{ij}\}_{i\in \sigma_j, \, j\in I_n}$ is called a {\it weaving}.

If we do not require the existence of the uniform frame bounds $A$, $B$  for all the weavings, we say that the family is {\it weakly woven}.
\end{definition}

In \cite{BCGLL}, the authors make an intensive study of woven frames.  
Among others, one of the most relevant results of that work states that weakly woven pairs are woven (\cite[Thm. 4.5]{BCGLL}):
\begin{theorem}\label{equiv ww and w}
Given two frames $\cF=\{f_i\}_{i\in I}$ and $\cG=\{g_i\}_{i\in I}$ for $\hil$, the following are equivalent:
\begin{itemize}
\item[(i)] The two frames are woven.
\item[(ii)] The two frames are weakly woven.
\end{itemize}
\end{theorem}

We shall restrict ourselves to the study of {\it woven pairs} of frames $(\cF,\cG)$ for $\hil$. That is, $\cF=\{f_i\}_{i\in I}$ and $\cG=\{g_i\}_{i\in I}$ are frames for $\hil$ and  there exist $A,B>0$ such that, for every partition $\{\sigma,\sigma^c\}$ of $I$,
\[ A\|x\|^2 \leq \sum_{i\in \sigma}\, |\pint{x}{f_i}|^2+\sum_{i\in \sigma^c}\, |\pint{x}{g_i}|^2\leq B\|x\|^2\, , \quad \text{ for every } x\in \hil.\]
  
\section{Woven  pairs and perturbations.}
 Let $\cF=\{f_i\}_{i\in I}$ be a frame for $\hil$ with synthesis operator $T_{\cF}$. It is obvious that the pair $(\cF,\cF)$ is a woven pair for $\hil$. A natural  question that arises is to determine if the pair $(\cF,\cG)$ is a woven pair as long as   $\cG$ is a frame ``sufficiently close'' to $\cF$.

In \cite{BCGLL}, the authors answer this question by showing that for a small perturbation $\cG$  of $\cF$, the pair is woven (see \cite[Thm. 6.1]{BCGLL}).  Our purpose is to  refine that result, showing that for every Bessel sequence $\cG=\{g_i\}_{i\in I}$ whose synthesis operator $T_{\cG}$  lies in an appropriate neighborhood of $T_{\cF}$,  the pair $(\cF,\cG)$ is woven.

Before stating the main result, we need the following well known result, which can be seen as a particular case of \cite[Thm. 3.1]{DH} about perturbations of a closed range operator: 
\begin{lemma}\label{Neumann}
Let $A\in B(\hil)$ be a surjective operator, with Moore-Penrose pseudoinverse $A^\dagger$. Let $B\in B(\hil)$ such that $\|A-B\|<\|A^\dagger\|^{-1}$, then $B$ is also surjective, with $\|B^\dagger\|<\frac{1}{\|A^\dagger\|^{-1}-\|A-B\|}$.

\end{lemma}

\begin{theorem}\label{epa}
Let  $\cF=\{f_i\}_{i\in I}$ be a frame for $\hil$ and  let  $\cG=\{g_i\}_{i\in I}$ be a Bessel sequence in $\hil$ with synthesis operators $T_{\cF}$ and $T_{\cG}$,  respectively. If $A_{\cF}$ is the optimal lower frame bound for $\cF$ and  $||T_{\cF}-T_{\cG}||^2<A_\cF$, then $(\cF,\cG)$ is a woven pair,  with frame bounds $(A_\cF^{1/2}-\|T_\cF-T_\cG\|)^2$ and $B_\cF+B_\cG$.
\end{theorem}

\begin{proof}
Each weaving  is a Bessel sequence with upper bound $B_\cF+B_\cG$. Indeed, for every $\sigma\subset I$ and $x\in \hil$ we have:
\[\sum_{i\in \sigma^c} |\pint{x}{f_i}|^2+\sum_{i\in \sigma}|\pint{x}{g_i}|^2\leq\sum_{i\in I} |\pint{x}{f_i}|^2+\sum_{i\in I}|\pint{x}{g_i}|^2 \leq (B_\cF+B_\cG)\|x\|^2.\]
 Recall that we denote by $\cB=\{e_i\}_{i\in I}$ the (fixed) orthonormal basis of $\hil$ such that $T_\cF\, e_i=f_i$, and $T_\cG\, e_i=g_i$, $\forall i\in I$.

Given a partition $\sigma\cup \sigma^c=I$, denote by $P_\sigma$ the orthogonal projection onto the closed span of $\{e_i\}_{i\in \sigma}$. 

Then, it is easy to see that the synthesis operator for the (Bessel) weaving $\cW_\sigma=\{f_i\}_{i\in \sigma^c}\cup \{g_i\}_{i\in \sigma}$ is the bounded operator 
\[T_{\cW_\sigma}=T_\cF(I-P_\sigma)+T_\cG P_\sigma=T_\cF+(T_\cG-T_\cF)P_\sigma.\]
Now, since 
\[ \|T_\cF-T_{\cW_\sigma}\|=\|(T_\cG-T_\cF)P_\sigma\|\leq \|T_\cF-T_\cG\|<A_{\cF}^{1/2}=\|T_{\cF}^\dagger\|^{-1},\]
we  have, by Lemma \ref{Neumann} that $T_{\cW_\sigma}$ is a surjective operator in $B(\hil)$. Moreover, we have  $\|T_{\cW_\sigma}^\dagger\|<(A_{\cF}^{1/2}-\|(T_{\cF}-T_{\cG})P_\sigma\|)^{-1}$.

Thus, we conclude, by Proposition \ref{equiv frames}, that  $\cW_\sigma$ is a frame for $\hil$ with optimal lower frame bound $A_{\cW_\sigma}\geq (A_\cF^{1/2}-\|(T_\cF-T_\cG)P_\sigma\|)^{2}$.

Finally,   since  $A_\cF^{1/2}-\|T_\cF-T_\cG\|\leq A_\cF^{1/2}-\|(T_\cF-T_\cG)P_\sigma\|$ for every $\sigma$, we conclude that $A=(A_\cF^{1/2}-\|T_\cF-T_\cG\|)^2$ is a uniform lower frame bound, therefore $(\cF, \cG)$ is a woven pair with frame bounds $A$ and $B=B_\cF+B_\cG$.
\end{proof}

The next corollary is a slight generalization of Proposition 6.2 in \cite{BCGLL}:
\begin{corollary}\label{invertible}Let $\cF=\{f_i\}_{i\in I}$ be a frame for $\hil$ with bounds $A_\cF, B_\cF$ and let $T\in B(\hil)$. Suppose that
\[ \|(Id-T)T_\cF\|<A_\cF^{1/2},\]
then, $\cF$ and $T(\cF)=\{Tf_i\}_{i\in I}$ are woven.
\end{corollary}
\begin{proof}We only have to notice that $TT_\cF$ is the synthesis operator for the Bessel sequence $\{Tf_i\}_{i\in I}$. The result then follows from the previous theorem.
\end{proof}

\begin{example}
Let $\{e_i\}_{i\in I}$ be an orthonormal basis of $\hil$. Let $M>1$ and consider the frame for $\hil$:
\[\cF=\{\sqrt{M}e_1,\; e_2,\;e_3,\ldots\}\]
with frame bounds $A=1$, $B=M$. Notice that, according \cite[Prop. 6.2]{BCGLL}, for those invertible operators $T$ such that $\|Id-T\|^2<\frac{A}{B}=\frac{1}{M}$, we have that $(\cF,T(\cF))$ is woven. In particular, for large upper bounds $M$, $T$ is too close to $Id$.

Now, let $0<c<1$ and 
$$T=e_1\otimes e_1+\sum_{i>1}(1-c)e_i\otimes e_i;$$
clearly, $T$ is invertible in $B(\hil)$ and $\|(Id-T)T_\cF\|=c<1$. Therefore, by Cor. \ref{invertible} $(\cF, T(\cF))$ is woven. Notice that, in this case,  $\|Id-T\|=c$ is as close to 1 as we want, regardless of the value of $M$.

\end{example}

Our next result gives us a sufficient condition for a linear perturbation of a woven pair by invertible operators to be woven.

Let $(\cF,\cG)$ be a woven pair of frames. Let us take some of the notations used in the proof of Thm. \ref{epa}. Hence, by   $\cW_\sigma$, we denote the weaving $\cW_\sigma=\{f_i\}_{i\in \sigma^c}\cup \{g_i\}_{i\in \sigma}$, whose lower and upper frame constants are $A_\sigma$ and $ B_\sigma$, respectively.
As we saw before, the synthesis operator of $\cW_\sigma$ is 
\[ T_{\cW_\sigma}=T_\cF+(T_\cG-T_\cF)P_\sigma.\]
(Recall that $P_\sigma$ denotes the orthogonal projection onto the closed subspace generated by $\{e_i\}_{i\in \sigma}$).

\begin{theorem} Let $\cF=\{f_i\}_{i\in I}$ and $\cG=\{g_i\}_{i\in I}$ be frames for $\hil$ with
 frame bounds $A_\cF,B_\cF$ and $A_\cG, B_\cG$ respectively, such that $(\cF,\cG)$ is a woven pair with optimal lower frame constant $0<C$. Let $U,V\in B(\hil)$ be a pair of invertible operators such that
\[ \|U^{-1}V-Id\|^2<\frac{C}{B_\cG} \quad \text{ or } \quad \|V^{-1}U-Id\|^2<\frac{C}{B_\cF}\]
then $U(\cF)=\{Uf_i\}_{i\in I}$ and $V(\cG)=\{Vg_i\}_{i\in I}$ are woven.

\end{theorem}
\begin{proof}
As we did before, our goal is to prove that the synthesis operator of the Bessel sequence  $\{Uf_i\}_{i\in \sigma^c}\cup \{Vg_i\}_{i\in \sigma}$, i.e.
\[ UT_\cF+(VT_\cG-UT_\cF)P_\sigma\]
is surjective, for all  $\sigma \subset I$ .

Suppose that $\|U^{-1}V-Id\|^2<\frac{C}{B_\cG}$, the other case is similar.

Let $\cW_\sigma$ be the weaving $\cW_\sigma=\{f_i\}_{i\in \sigma^c}\cup \{g_i\}_{i\in \sigma}$. Then, by definition of woven frames we have that $C\leq A_\sigma=\|T_{\cW_\sigma}^\dagger\|^{-2}$.

Since by assumption $\|U^{-1}V-Id\|^2<\frac{C}{B_\cG}$, we have
\begin{align*}
\|U^{-1}VT_\cG-T_\cG\|&\leq\|U^{-1}V-Id\|\ \|T_\cG\|\\ 
                   &\leq \|U^{-1}V-Id\| \ B_\cG^{1/2}<C^{1/2}\leq \|T_{\cW_\sigma}^\dagger\|^{-1}.
\end{align*}

Then,
\[\|T_{\cW_\sigma}-(T_\cF+(U^{-1}VT_\cG-T_\cF)P_\sigma)\|=\|(U^{-1}VT_\cG-T_\cG)P_\sigma\|<         \|T_{\cW_\sigma}^\dagger\|^{-1}.\]
Hence, by Lemma \ref{Neumann}, the operator $T_\cF+(U^{-1}VT_\cG-T_\cF)P_\sigma$ is surjective, so do is $U(T_\cF+(U^{-1}VT_\cG-T_\cF)P_\sigma)$, therefore, the weaving $\{Uf_i\}_{i\in \sigma^c}\cup \{Vg_i\}_{i\in \sigma}$ is a frame for $\hil$.

Since $\sigma \subset I$ is arbitrary, the pair $(U(\cF), V(\cG))$ is weakly woven, so it is a woven pair by the equivalence between weakly woven and woven pairs of frames.
            
\end{proof}

\begin{remark}
The necessary condition given in the previous theorem is not sufficient. To see this, let us consider a finite dimensional Hilbert space $\hil$ and let $\cF=\{e_i\}_{i\in I}$ be a orthonormal basis for $\hil$. It is clear that $(\cF,\cF)$ is a woven pair with frame constants $A=B=1$. On the other hand, it easy to see that,  for every invertible operator $V$, diagonal with respect to $\cF$ (that is, such that   $Ve_i=\alpha_i \,e_i$, for some $\alpha_i\neq 0$,  $\forall i\in I$), the pair $(\cF, V(\cF))$ is woven. 

In particular, if we take $V$ such that $\|V-Id\|\geq 1$ and $\|V^{-1}-Id\|\geq 1$ we see that the condition  is not satisfied for $U=Id$ even though $U(\cF)$ and $V(\cF)$ are woven.
\end{remark}

As a consequence of this result we can obtain the next result about canonical duals:

\begin{corollary}Let $(\cF,\cG)$ a woven pair as before, with lower frame constant $0<C$. Denote by $S_\cF$ and $S_\cG$ the frame operators for $\cF,\cG$,  respectively. Then if
\[\|S_\cF^{-1}-S_\cG^{-1}\|^2<C\,\max\{\frac{1}{B_\cG B_\cF^2},\, \frac{1}{B_\cF B_\cG^2}\}\]
we have that the canonical duals $\{S_\cF^{-1}f_i\}_{i\in I}$ and $\{S_\cG^{-1}g_i\}_{i\in I}$ form a woven pair.
\end{corollary}
\section{A characterization of woven pairs.}
In this section we consider a connection between woven pairs and Riesz Frames. More specifically we want to derive a necessary condition for a woven pair in terms of angle (or gap) between certain closed subspaces that resembles the condition found for Riesz frames in \cite{[ACSR]}. 

Let us recall the definition of Riesz Frames, introduced by O. Christensen in \cite{Chris} :

\begin{definition}
A frame  $\cF=\{f_i\}_{i\in I}$ for $\hil$ is called a {\sl Riesz frame} if there exists $A,B>0$ such that, for every $\sigma\subset I$ the sequence $\{f_i\}_{i\in \sigma}$ is a frame sequence with bounds $A,B$. 
\end{definition}

In \cite{[ACSR]} the authors show that the Riesz frames can be characterized in terms of the nullspace of the synthesis operator, considering its position with respect to ``diagonal'' subspaces, that is, those closed subspaces generated by $\{e_i\}_{i\in \sigma}$. 

Specifically, the uniform lower bound for the (sub)frame sequences of a Riesz frame is related to the existence of a uniform bound $0<\beta$ for the angle between the nullspace of the synthesis operator and the diagonal subspaces. 

Let us introduce some definitions first:
\begin{definition}\label{angulos}
Given two closed subspaces $M$ and $N$ of a Hilbert space $\hil$, let $\tilde{M}=M\ominus (M\cap N)$ and $\tilde{N}=N\ominus (M\cap N)$. The {\it angle}
between $M$ and $N$ is the angle in $[0,\pi/2]$ whose cosine
is defined by
\begin{align*}
\angf{M}{N}&=\sup\{\,|\pint{x}{y}|:\;x\in \tilde{M}, \;y\in \tilde{N}
\;\mbox{and}\;\|x\|=\|y\|=1 \}\\
&=\|P_MP_NP_{(M\cap N)^\perp}\|
\end{align*}

\end{definition}

The  angle $\angf{M}{N}$ is  related with the {\it gap} $\delta(M,N)$ between the closed subspaces $M$, $N$:
\[\delta(M, N)=\sup_{x\in M, \, \|x\|=1} \, \dist(x,N)=\|P_{N^\perp}P_M\|\]

In particular, if $M\cap N^\perp=\{0\}$, $\delta(M,N)=\angf{M}{N^\perp}$.

The following  result shows the closed connection between  angles and closed range
 operators. For more details and properties we refer the reader to the  work by  Deutsch \cite{[De]}.
\begin{proposition}\label{producto con rango cerrado}
Let $A,B\in B(\hil)$ be closed range operators. Then, $AB$ has closed range if and only if
$\angf{R(B)}{\\N(A)}<1$.
\end{proposition}

For an operator $T\in B(\hil)$ we  define its {\it reduced minimum modulus} by

\begin{equation}\label{gamma}
\gamma (T):= \inf \{ \|Tx\| \; | \;  \|x\|=1 \; , \;  x\in N(T)^\perp \}.
\end{equation}

\noindent
It is well known that $T$ has closed range if and only if $\gamma
(T)>0$. Moreover, in this case, $\gamma (T)=\gamma (T^*) = \gamma(T^*T)^{1/2}=\|T^\dag \|^{-1}$.

Finally, the following result relates all these concepts, and will be useful in our characterization of woven pairs:

\begin{proposition}\label{gammas}(Rem. 2.10, \cite{[ACSR]})
Let $A, B\in B(\hil)$ with closed ranges
and let $ c:= \angf{N(A)}{R(B)}$.
 Then,
\begin{equation}\label{desigualdad}
 \gamma (A)\gamma (B) (1-c^2)^{1/2} \leq \gamma (AB)\leq \|A\|\|B\| (1-c^2)^{1/2}.
\end{equation}
In particular, $AB$ is a closed range operator if and only if $c<1$.
\end{proposition}

Now, if we have a  frame $\cF=\{f_i\}_{i\in I}$, with synthesis operator $T_\cF$, it turns out that $\cF$ is a Riesz frame if and only if $N(T_\cF)$ (the nullspace of $T_\cF$) is {\it compatible} with respect to the basis $\cB=\{e_i\}_{i\in I}$, that means that:
\[ \sup_{\sigma \subset I} \angf{N(T_\cF)}{R(P_\sigma)}<1\]
(see \cite{[ACSR]}).

Our purpose in this section is to find a similar condition to a woven pair. It is clear that an equivalent condition in terms of angles is not possible, since each weaving must be a frame for $\hil$ and not just a frame sequence. However, by properly defining the setting, we can state a sufficient condition for a woven pair. 

Suppose that we have a woven pair $(\cF,\cG)$ for $\hil$. Denote by $\tilde{\hil}=\hil \oplus \hil$ the Hilbert space endowed with the inner product
$$
\pint{x\oplus y}{z\oplus w}_{\tilde{\hil}}=\pint{x}{z}+\pint{y}{w}
.$$
Also, for each $\sigma \subset I$ we shall denote by $Q_\sigma$ to the oblique projection in $B(\tilde{\hil})$ given by
$$Q_{\sigma}=\left(
\begin{array}{cc}
I & P_{\sigma}\\
0 & 0\\
\end{array}
\right).$$

\begin{theorem}\label{equiang}
Let  $\cF=\{f_i\}_{i\in I}$ and $\cG=\{g_i\}_{i\in I}$ be frames for $\hil$. Define  the bounded operator $T_{\cF,\cG}:\tilde{\hil} \rightarrow \hil$ by
\[T_{\cF,\cG}(x\oplus y)=
T_\cF(x)+  
(T_\cG-T_\cF)(y).\] 
Therefore, if $(\cF,\cG)$ is a woven pair, then
\begin{equation}\label{compatibilidad torcida}
 \sup_{\sigma\subset I} \angf{N(T_{\cF,\cG})}{R(Q_{\sigma}^*)} < 1. 
\end{equation}

\end{theorem}

\begin{proof}
Denote by $B=T_\cF^*$ and $A=T_\cG^*-T_\cF^*$. Then, if $(\cF,\cG)$ is a woven pair, it is clear that there exists a constant $C>0$ such that, for every $\sigma \subset I$:
\[C\|x\|\leq \|(P_{\sigma}A+B) \ x\| \qquad \forall x\in \hil.\]
It is easy to see that $T_{\cF,\cG}^*$ is an injective closed range operator (since $T_\cF$ and $T_\cG$ are synthesis operators of frames).

On the other hand, if $x\in \hil$, $\|x\|=1$,
$$||T_{\cF,\cG}^*x||=\sqrt{||Bx||^2+||Ax||^2} \geq ||Bx|| \geq \gamma(B)=\gamma(T_\cF^*)=A_\cF^{1/2}$$
so $\gamma(T_{\cF,\cG}^*)\geq A_\cF^{1/2}$.

Finally, we have 
$$\gamma(Q_{\sigma}T_{\cF,\cG}^*)=\gamma\left(\left(
\begin{array}{c}
(B+P_{\sigma}A) \ \ 0\\
\end{array}
\right)^\intercal\right)=\gamma(B+P_{\sigma}A).$$

Therefore, if $c=\angf{N(Q_{\sigma})}{R(T_{\cF,\cG}^*}$, by  Prop. \ref{gammas}:
\begin{equation}
\gamma(Q_{\sigma})\gamma(T_{\cF,\cG}^*) (1-c^{1/2})^{1/2}\leq \gamma(B+P_{\sigma}A)\leq ||Q_{\sigma}||\ ||T_{\cF,\cG}^*|| \  (1-c^{1/2})^{1/2}.
\end{equation}
Thus, if  $(\cF,\cG)$ woven, the uniform lower bound for $\gamma(B+P_{\sigma}A)$ implies  
$$\sup_{\sigma\subset I}\angf{N(Q_{\sigma})}{R(T_{\cF,\cG}^*}=\sup_{\sigma\subset I}\angf{N(T_{\cF,\cG})}{R(Q_{\sigma}^*)} <1.$$ 

\end{proof}
\begin{remark}The condition \eqref{compatibilidad torcida}, which works as a kind of ``oblique'' compatibility condition between  $N(T_{\cF,\cG})$ and the ranges of $Q_\sigma^*$,  does not guarantee that $(\cF,\cG)$ is a woven pair since it only implies that the weavings $\{f_i\}_{i\in \sigma^c}\cup \{g_i\}_{i\in \sigma}$ are frame sequences for $\hil$ with a uniform lower frame bound for every $\sigma$. 
\end{remark}

\medskip

\section{Woven pairs for  the scaled canonical dual  frame.}

In \cite[Cor. 5.4]{LLHL}, the authors state a result that provides a sufficient condition on a g-frame and a scaled canonical dual g-frame in order to be a woven pair. In terms of usual vector frames, it can be written as:

\begin{proposition} 
Let $\cF=\{f_i\}_{i\in I}$ a frame for $\hil$ with bounds $A_\cF,B_\cF$. If $\frac{B_\cF}{A_\cF}<2$, then $\cF$ and the scaled  canonical dual frame $\cG=\{\frac{2A_\cF B_\cF}{A_\cF+B_\cF}\, S_\cF^{-1}f_i\}_{i\in I}$ are a woven pair.
\end{proposition}

The purpose of this section is to apply Thm. \ref{epa} to this particular case. Specifically, we shall show that for a greater ratio $\frac{B_\cF}{A_\cF}$ we can determine an interval of positive numbers $\alpha$ which guarantee that the pair $(\cF,\alpha\cdot \cF^\sharp)$ is woven, where we denote $\cF^\sharp$ to the canonical dual  frame of $\cF$.

\begin{theorem}\label{menos4}
Let $\cF=\{f_i\}_{i\in I}$ a frame for $\hil$ with bounds $0<A_{\cF}\leq B_{\cF}$. If $\frac{B_{\cF}}{A_{\cF}}<4$, then for every $\alpha>0$ which satisfies 
$B_{\cF}-\sqrt{A_{\cF}B_{\cF}}<\alpha<2A_{\cF}$, $(\cF,\alpha\cdot \cF^\sharp)$ is a woven pair.
\end{theorem}
\begin{proof} 
According to Thm. \ref{epa} we need to prove that 
under our hypothesis, we have that
$$
 \|T_\cF-\alpha S_\cF^{-1}T_\cF\|^2<A_{\cF}.
$$
If we rewrite the left side of this inequality, 
\[
\|T_\cF-\alpha S_\cF^{-1}T_\cF\|^2=
        \|(I-\alpha S_\cF^{-1})T_\cF T_\cF^*(I-\alpha S_\cF^{-1})\|=\|(S_\cF-\alpha I) S_\cF^{-1} (S_\cF-\alpha I)\|.
\]
Using functional calculus, we can obtained specifically that
\[\|(S_\cF-\alpha I) S_\cF^{-1} (S_\cF-\alpha I)\|=\max_{x\in [A_\cF,B_\cF]} \, \frac{(x-\alpha)^2}{x}.\]
In particular, 
\[ \|T_\cF-\alpha S_\cF^{-1}T_\cF\|^2=\max \left \{ \frac{(B_\cF-\alpha)^2}{B_\cF},\;\frac{(A_\cF-\alpha)^2}{A_\cF} \right\}.
\]
As we supposed $B_\cF<4A_\cF$, then $B_\cF-\sqrt{A_\cF B_\cF}<\sqrt{A_\cF B_\cF}$. Let $\alpha$ be such that $$B_\cF-\sqrt{A_\cF B_\cF}<\alpha\leq \sqrt{A_\cF B_\cF}.$$
It is easy to check that $\alpha\leq \sqrt{A_\cF B_\cF}$ implies that 
\[  \|T_\cF-\alpha S_\cF^{-1}T_\cF\|^2= \frac{(B_\cF-\alpha)^2}{B_\cF}.\]
In the other hand, from $B_\cF-\sqrt{A_\cF B_\cF}<\alpha$, we get that
$$\frac{(B_\cF-\alpha)^2}{B_\cF}<\frac{\left(B_\cF-(B_\cF-\sqrt{A_\cF B_\cF})\right)^2}{B_\cF}=A_\cF,$$ 
and hence the pair is woven.

Similarly, if $\sqrt{A_\cF B_\cF}<\alpha<2A_\cF$, it turns out that 
\[\|T_\cF-\alpha S_\cF^{-1}T_\cF\|^2= \frac{(A_\cF-\alpha)^2}{A_\cF}<A_\cF,\]
and together with Thm. \ref{epa}, we can conclude the proof.
\end{proof}

\begin{remark} The value of the scale used in the cited work \cite{LLHL}, i.e. 
$\alpha=\frac{2A_{\cF}B_{\cF}}{A_{\cF}+B_{\cF}}$, is not necessarily under the conditions of our result. It is clear that, because of the arithmetic-geometric mean inequality, $\alpha$ is less than or equal to $\sqrt{A_{\cF}B_{\cF}}$. However, the condition $\frac{B_{\cF}}{A_{\cF}}<4$ does not ensure that $B_{\cF}-\sqrt{A_{\cF}B_{\cF}}<\alpha$ in this case. In the following we see this carefully.

Due to $\alpha\leq \sqrt{A_{\cF}B_{\cF}}$, we need to find conditions over the ratio $\frac{B_\cF}{A_\cF}$ in order to have
$$B_\cF-\sqrt{A_\cF B_\cF}<\alpha$$
and hence arrive to the conclusion of Thm.\ref{menos4}. If we make a short calculation, we see that the condition
$$B_\cF-\frac{2A_\cF B_\cF}{A_\cF+B_\cF}\leq \sqrt{A_\cF B_\cF}$$ 
is equivalent to have
\[\frac{B_\cF-A_\cF}{B_\cF+A_\cF}<\sqrt{\frac{A_\cF}{B_\cF}}.\]
Moreover, if we call $r=\frac{B_\cF}{A_\cF}$, this inequality turns into
\[\frac{(r-1)^2}{(r+1)^2}<\frac{1}{r}.\]
In particular, our interest is to find $1\leq r$ such that $r^3-3r^2-r-1<0$. As the polynomial $f(r)=r^3-3r^2-r-1$ has a unique real root
\[ r_0=\left( \sqrt[3]{\frac{3+\sqrt{\frac{11}{3}}}{6}}+\sqrt[3]{\frac{3-\sqrt{\frac{11}{3}}}{6}}\right)^3\approx 3.383,\]
then we can allow a greater bound for $\frac{B_\cF}{A_\cF}$ (that is $\frac{B_\cF}{A_\cF}<3.383$) to arrive to the same result as in \cite{LLHL}.

\end{remark}


\begin{thebibliography}{}

\bibitem{[ACSR]}  J. Antezana, G. Corach, D. Stojanoff and M. Ruiz, {\it Weighted projections and Riesz frames}, Lin. Alg. Appl. 402,  367--389, (2005).

\bibitem{BCGLL} T. Bemrose, P. G. Casazza, K.  Gröchenig, M. C.  Lammers, R. G. Lynch,{\it Weaving frames}, Oper. Matrices {\bf 10}, 4, 1093--1116,(2016). 

\bibitem{CFL} P. G. Casazza, D. Freeman, R. G. Lynch,
{\it Weaving Schauder frames},
Journal of Approximation Theory,
{\bf 211}, 42-60,
(2016).

\bibitem{CL}  P. G. Casazza and R. G. Lynch, {\it Weaving properties of Hilbert space frames}, Proc. SampTA , 110--114,(2015).

\bibitem{Chris} O. Christensen, {\it Frames containing a Riesz basis and approximation of the frame coefficients using finite dimensional
methods}, J. Math. Anal. Appl. 199,256--270, (1996).

\bibitem{DeV1} Deepshikha and L. K. Vashisht, {\it On Continuous weaving frames},   Advances  in  Pure and  Applied  Mathematics,   8 (1),  15-31, (2017).

\bibitem{DeV2} Deepshikha and L. K. Vashisht,{\it Weaving K- Frames in Hilbert Spaces}, Results in Mathematics, 73-81, (2018).

\bibitem{DeVV} S. G. Deepshikha, L. K. Vashisht and G. Verma, {\it On weaving fusion frames for Hilbert spaces}, 2017 International Conference on Sampling Theory and Applications (SampTA), Tallin, 381–385,(2017). doi: 10.1109/SAMPTA.2017.8024363.

\bibitem{[De]}  F. Deutsch, {\it The angle between subspaces in Hilbert
space}, Approximation theory, wavelets and applications, Netherlands,  107--130,(1995).

\bibitem{DH}  J. Ding and L. J. Huang, {\it Perturbation of generalized inverses of linear operators in Hilbert spaces}, J. Math.
Anal. Appl. 198, 505--516, (1996) .

\bibitem{LLHL} Li, D., Leng, J., Huang, T.,  Li, X.,{\it On Weaving g-Frames for Hilbert Spaces}, Complex Analysis and Operator Theory, 14(2),(2020). doi:10.1007/s11785-020-00991-7 

\bibitem{NA} F. A. Neyshaburi and A. A. Arefijamaal, {\it Weaving Hilbert space fusion frames}, preprint, arXiv: 1802.03352.


\end{thebibliography}
\end{document}